\documentclass[12pt]{amsart}
\usepackage{mathrsfs}
\usepackage{amssymb}
\usepackage[hang,flushmargin]{footmisc}
\usepackage{multirow}
\usepackage{amsfonts}
\usepackage{graphicx}
\usepackage{amsmath,amsthm,amssymb,latexsym,color}
 \makeatletter
    
    \newcommand{\Rmnum}[1]
    {\expandafter\@slowromancap\romannumeral #1@}
    \makeatother
\def\wz{\tilde}

\newtheorem{thm}{Theorem}[section]

\newtheorem{lemma}[thm]{Lemma}
\newcounter{foo}[subsection]

\newtheorem{step}[foo]{Step}

\newtheorem{example}[thm]{Example}
\newtheorem{defin}[thm]{Definition}

\theoremstyle{definition}
\newtheorem{remark}[thm]{Remark}

\title{Primitive weakly distance-regular circulant digraphs}

\author{Akihiro Munemasa}
\author{Kaishun Wang}
\author{Yuefeng Yang}
\address{Research Center for Pure and Applied Mathematics\\
Graduate School of Information Sciences\\
Tohoku University, Sendai, 980--8579, Japan}
\email{munemasa@math.is.tohoku.ac.jp}
\address{Sch. Math. Sci. \& Lab. Math. Com. Sys.\\
 Beijing Normal University \\
 Beijing, 100875, China }
\email{wangks@bnu.edu.cn}
\address{School of Science \\
China University of Geosciences \\
Beijing, 100083, China}
\email{yangyf@cugb.edu.cn}

\begin{document}
\date{}
\begin{abstract}
We classify certain non-symmetric commutative association schemes. As an application, we determine all the primitive weakly distance-regular circulant digraphs.
\end{abstract}

\maketitle
%\footnote{\scriptsize
%{\em E-mail address:} 
%munemasa@math.is.tohoku.ac.jp (Akihiro Munemasa), wangks@bnu.edu.cn (Kaishun Wang), yangyf@cugb.edu.cn (Yuefeng Yang).}

\section{Introduction}

A \emph{$d$-class association scheme} $\mathfrak{X}$ is a pair $(X,\{R_{i}\}_{i=0}^{d})$, where $X$ is a finite set, and each $R_{i}$ is a
nonempty subset of $X\times X$ satisfying the following axioms (see \cite{EB84,PHZ96,PHZ05} for a background of the theory of association schemes):
\begin{itemize}
\item [{\rm(i)}] $R_{0}=\{(x,x)\mid x\in X\}$ is the diagonal relation;

\item [{\rm(ii)}] $X\times X=R_{0}\cup R_{1}\cup\cdots\cup R_{d}$, $R_{i}\cap R_{j}=\emptyset~(i\neq j)$;

\item [{\rm(iii)}] for each $i$, $R_{i}^{T}=R_{i^{*}}$ for some $0\leq i^{*}\leq d$, where $R_{i}^{T}=\{(y,x)\mid(x,y)\in R_{i}\}$;

\item [{\rm(iv)}] for all $i,j,l$, the cardinality of the set $$P_{i,j}(x,y):=R_{i}(x)\cap R_{j^{*}}(y)$$ is constant whenever $(x,y)\in R_{l}$, where $R(x)=\{y\mid (x,y)\in R\}$ for $R\subseteq X\times X$ and $x\in X$. This constant is denoted by $p_{i,j}^{l}$.
\end{itemize}
A $d$-class association scheme is also called an association scheme with $d$ classes (or even simply a scheme). The integers $p_{i,j}^{l}$ are called the \emph{intersection numbers} of $\mathfrak{X}$. We say that $\mathfrak{X}$ is \emph{commutative} if $p_{i,j}^{l}=p_{j,i}^{l}$ for all $i,j,l$. The subsets $R_{i}$ are called the \emph{relations} of $\mathfrak{X}$. For each $i$, the integer $k_{i}$ $(=p_{i,i^{*}}^{0})$ is called the \emph{valency} of $R_{i}$.   A relation $R_{i}$ is called \emph{symmetric} if $i=i^{*}$, and \emph{non-symmetric} otherwise. An association scheme is called \emph{symmetric} if all relations are symmetric, and \emph{non-symmetric} otherwise. An association
scheme is called \emph{skew-symmetric} if the diagonal relation is the
only symmetric relation.

Let $\mathfrak{X}=(X,\{R_{i}\}_{i=0}^{d})$ be an association scheme with $|X|=n$. For two nonempty subsets $E$ and $F$ of $\{R_{i}\}_{i=0}^{d}$, define
\[EF:=\{R_{l}\mid\sum_{R_{i}\in E}\sum_{R_{j}\in F}p_{i,j}^{l}\neq0\}.\]
We write $R_{i}R_{j}$ instead of $\{R_{i}\}\{R_{j}\}$,
and $R_i^2$ instead of $\{R_{i}\}\{R_{i}\}$.
If $R_{i^{*}}R_{j}\subseteq F$ for any $R_{i},R_{j}\in F$, we say that $F$ is {\em closed}. We say that $R_{j}$ {\em generates} $\mathfrak{X}$ if the smallest
closed subset containing $R_{j}$ is equal to $\{R_{i}\}_{i=0}^{d}$. We call $\mathfrak{X}$ \emph{primitive} if every non-diagonal relation generates $\mathfrak{X}$. Otherwise, $\mathfrak{X}$ is said to be \emph{imprimitive}.

In this paper, we classify certain non-symmetric commutative association schemes. Our first main result is as follows. See Section 2 for precise definitions of a pseudocyclic scheme and cyclotomic scheme.

\begin{thm}\label{main1}
Let $\mathfrak{X}=(X,\{R_{i}\}_{i=0}^{d})$ be a commutative  association scheme generated by a non-symmetric relation $R_{1}$ satisfying
\begin{align}
R_{1}^{2}&\subseteq\{R_{1},R_{1^{*}},R_{2}\},\label{1.1a}\\
R_{1}R_{1^{*}}&\subseteq\{R_{0},R_{1},R_{1^{*}},R_{2},R_{2^{*}}\},\label{1.1b}\\
2&\notin\{1^{*},2^{*}\} .\label{1.1c}
\end{align}
If $k_{1}=k_{2}>1$, then $d=4$. Moreover, if $\mathfrak{X}$ is pseudocyclic, then $\mathfrak{X}$ is isomorphic to the cyclotomic scheme ${\rm Cyc}(13,4)$.
\end{thm}

\begin{remark}
Let $\mathfrak{X}=(X,\{R_{i}\}_{i=0}^{d})$ be a commutative  association scheme generated by a non-symmetric relation $R_{1}$. If $k_{1}=1$, then  $\mathfrak{X}$ is isomorphic to the group scheme over a cyclic group of order greater than $2$. For a definition of a group scheme, see {\rm\cite[Chapter \Rmnum{2}, Example 2.1 (2)]{EB84}}.
\end{remark}

As a natural directed version of distance-regular graphs (see \cite{AEB98,DKT16} for a background of the theory of distance-regular graphs), Wang and Suzuki \cite{KSW03} introduced the concept of weakly distance-regular digraphs. In \cite{SM03}, Miklavi\'c and Poto\v cnik gave the classification of distance-regular circulants. In this paper, we show that the attached schemes of primitive weakly distance-regular circulant digraphs satisfy the conditions of Theorem~\ref{main1}, and obtain the classification of such digraphs. In order to state our result, we introduce some basic notations and terminologies about weakly distance-regular circulant digraphs. See \cite{HS04,KSW03,KSW04,YYF16,YYF,YYF16+} for more details.

Let $\Gamma$ denote a finite simple digraph, which is not undirected. We write $V\Gamma$ and $A\Gamma$ for the vertex set and arc set of $\Gamma$, respectively. For any $x,y\in V\Gamma$, let $\partial(x,y)$ be the \emph{distance} from $x$ to $y$, and $\wz{\partial}(x,y):=(\partial(x,y),\partial(y,x))$ the \emph{two-way distance} from $x$ to $y$. An arc $(u,v)$ of $\Gamma$ is of \emph{type} $(1,r)$ if $\partial(v,u)=r$.

A strongly
connected digraph $\Gamma$ is said to be \emph{weakly distance-regular} if the configuration $\mathfrak{X}(\Gamma)=(V\Gamma,\{\Gamma_{\wz{i}}\}_{\wz{i}\in\wz{\partial}(\Gamma)})$ is an association scheme, where $\wz{\partial}(\Gamma)=\{\wz{\partial}(x,y)\mid x,y\in V\Gamma\}$ and $\Gamma_{\wz{i}}=\{(x,y)\in V\Gamma\times V\Gamma\mid\wz{\partial}(x,y)=\wz{i}\}$. We call $\mathfrak{X}(\Gamma)$ the \emph{attached scheme} of $\Gamma$. We say that $\Gamma$ is \emph{primitive} if $\mathfrak{X}(\Gamma)$ is primitive.

Let $H$ be a finite, multiplicatively written, group with identity $e$, and $S\subseteq H\setminus\{e\}$ be a generating set of $H$ containing an element $s$ such that $s^{-1}\notin S$. A \emph{Cayley digraph} of a group $H$ with respect
to the set $S$, denoted by $\textrm{Cay}(H,S)$, is the digraph with vertex set $H$, where
$x\in H$ is adjacent to $y\in H$ whenever $yx^{-1}\in S$. A digraph which is isomorphic to a Cayley digraph of a cyclic group is called a \emph{circulant digraph}.

Our second main theorem is as follows. See Section 2 for a definition of a Paley digraph.

\begin{thm}\label{Main2}
A digraph
$\Gamma$ is a primitve weakly distance-regular circulant digraph,
if and only if
$\Gamma$ is isomorphic to one of the following digraphs:
\begin{itemize}
\item [{\rm(i)}] the Paley digraph of order $p$, where $p$ is prime such that
$p\equiv 3\pmod4$.

\item [{\rm(ii)}] the circuit of length $p$, where $p$ is prime.

\item [{\rm(iii)}] ${\rm Cay}(\mathbb{Z}_{13},\{1,3,9\})$.
\end{itemize}
\end{thm}

\section{Preliminaries}

Let $\mathfrak{X}=(X,\{R_{i}\}_{i=0}^{d})$ be a commutative association scheme with $|X|=n$. The \emph{adjacency matrix} $A_{i}$ of $R_{i}$ is the $n\times n$ matrix whose $(x,y)$-entry is $1$ if $(x,y)\in R_{i}$ and $0$ otherwise. By the \emph{adjacency} or \emph{Bose-Mesner algebra} $\mathfrak{U}$ of $\mathfrak{X}$ we mean the algebra generated by $A_{0},A_{1},\ldots,A_{d}$ over the complex field. Axioms (i)--(iv) are equivalent to the following:
\[A_{0}=I,\quad \sum_{i=1}^{d}A_{i}=J,\quad A_{i}^{T}=A_{i^{*}},
\quad A_{i}A_{j}=\sum_{i=0}^{d}p_{i,j}^{l}A_{l},\]
where $I$ and $J$ are the identity and all-one matrices of order $n$, respectively.

Since $\mathfrak{U}$ consists of commuting normal matrices, it has a second basis consisting of primitive idempotents $E_{0}=J/n, E_{1},\ldots,E_{d}$. The integers $m_{i}=\textrm{rank}E_{i}$ are called the \emph{multiplicities} of $\mathfrak{X}$, and $m_{0}=1$ is said to be the trivial multiplicity. A commutative association scheme is called \emph{pseudocyclic} if all the non-trivial multiplicities coincide with each other.

Now we list basic properties of intersection
numbers.

\begin{lemma}[{\cite[Chapter \Rmnum{2}, Proposition 2.2]{EB84}}]\label{jb}
Let $(X,\{R_{i}\}_{i=0}^{d})$ be a commutative association scheme. The following hold:
\begin{itemize}
\item [{\rm(i)}] $k_{i}k_{j}=\sum_{l=0}^{d}p_{i,j}^{l}k_{l}$.

\item [{\rm(ii)}] $p_{i,j}^{l}k_{l}=p_{l,j^{*}}^{i}k_{i}=p_{i^{*},l}^{j}k_{j}$.

\item [{\rm(iii)}] $\sum_{j=0}^{d}p_{i,j}^{l}=k_{i}$.

\item [{\rm(iv)}] $\sum_{\alpha=0}^{d}p_{i,j}^{\alpha}p_{f,\alpha}^{l}=\sum_{\beta=0}^{d}p_{f,i}^{\beta}p_{\beta,j}^{l}$.
\end{itemize}
\end{lemma}

We call an association scheme $(X,\{R_{i}\}_{i=0}^{d})$ a \emph{translation scheme} if, the underlying set $X$ has the structure of an additive group,  and for all relations $R_{i}$,
\[(x,y)\in R_{i}\Longleftrightarrow(x+z,y+z)\in R_{i}\quad\textrm{for all $z\in X$}.\]

A classical example of translation schemes is the cyclotomic scheme which we describe now. Let $\textrm{GF}(q)$ be the finite field with $q$ elements, where $q$ is a prime power, and $\alpha$ be a primitive element of $\textrm{GF}(q)$. For a fixed divisor $d$ of $q-1$, define
\[(x,y)\in R_{i}\quad\textrm{if}~x-y\in\alpha^{i}\langle\alpha^{d}\rangle,~1\leq i\leq d.\]
These relations $R_{i}$ define a pseudocyclic association scheme, denoted by $\textrm{Cyc}(q,d)$, called the $d$-class \emph{cyclotomic scheme} over ${\rm GF}(q)$.

\begin{lemma}\label{lem:1}
Let $({\rm GF}(q),\{R_{i}\}_{i=0}^{d})$ be the $d$-class cyclotomic scheme. For $1\leq i\leq d$ and $s\in{\rm GF}(q)\setminus\{0\}$, define
\begin{align}
R_{i}^{(s)}=\{(sx,sy)\mid(x,y)\in R_{i}\}.\nonumber
\end{align}
Then, for any fixed $i\in\{1,2,\ldots,d\}$,
\begin{align}
\{R_{i}^{(s)}\mid s\in{\rm GF}(q)\setminus\{0\}\}=\{R_{j}\mid1\leq j\leq d\}.\nonumber
\end{align}
\end{lemma}
\begin{proof}
Let $\alpha$ be a primitive element of ${\rm GF}(q)$. Then
\begin{align}
\{s\alpha^{i}\langle\alpha^{d}\rangle\mid s\in{\rm GF}(q)\setminus\{0\}\}=\{\alpha^{j}\langle\alpha^{d}\rangle\mid1\leq j\leq d\},\nonumber
\end{align}
and the result follows immediately from this.
\end{proof}

As a consequence of Lemma~\ref{lem:1}, as digraphs, we have $(\textrm{GF}(q),R_i)\cong (\textrm{GF}(q),R_j)$
for $1\leq i,j\leq d$. For example when $(q,d)=(13,4)$ we have $(\textrm{GF}(13),R_i)\cong \textrm{Cay}(\mathbb{Z}_{13},\{1,3,9\})$ for $1\leq i\leq 4$.

\begin{lemma}\label{skew}
Let $\mathfrak{X}=(X,\{R_{0},R_{1},R_{1^{*}},R_{2},R_{2^{*}}\})$ be a $4$-class skew-symmetric and pseudocyclic association scheme. If $|X|>5$ and $p_{1,1}^{2^{*}}=0$, then $\mathfrak{X}$ is isomorphic to ${\rm Cyc}(13,4)$.
\end{lemma}
\begin{proof}
The symmetrization $(X,\{R_{0},R_{1}\cup R_{1^{*}},R_{2}\cup R_{2^{*}}\})$ of $\mathfrak{X}$ is a pseudocyclic association scheme of class $2$. Thus, $\mathfrak{X}$ is a skew-symmetric fission of a conference graph (see \cite{JM11} for definitions). By the proof of \cite[Theorem 3.3]{JM11}, there exist integers $u$ and $v$ such that
\begin{align}
|X|&=u^{2}+4v^{2},\label{skew1}\\
p_{1,1}^{2^{*}}&=\frac{|X|+1+2u\pm8v}{16}\nonumber.
\end{align}
Since $p_{1,1}^{2^{*}}=0$, we have $v^{2}=(|X|+1+2u)^{2}/64$. Substituting this into \eqref{skew1}, we obtain
\[20u^{2}+4(|X|+1)u+|X|^{2}-14|X|+1=0.\]
Since $u$ is an integer, the discriminant of the quadratic is not less than zero. Thus
\[(|X|+1)^{2}-5(|X|^{2}-14|X|+1)\geq0.\]
It follows that $9-4\sqrt{5}\leq |X|\leq9+4\sqrt{5}$. By \cite[Theorem 3.3]{JM11}, one gets $|X|\equiv 5~(\textrm{mod}~8)$, which implies $|X|=13$. In view of \cite[Result 1]{MH96}, $\mathfrak{X}$ is isomorphic to $\textrm{Cyc}(13,4)$.
\end{proof}

We close the section with a definition of a Paley digraph. For a prime power
$q=p^{m}$, $q\equiv 3~(\textrm{mod}~4)$, let $(\textrm{GF}(q),\{R_{i}\}_{i=0}^{2})$ denote the cyclotomic scheme $\textrm{Cyc}(q,2)$. The \emph{Paley digraph of order $q$} is defined as the digraph $(\textrm{GF}(q),R_{1})$. It is
an easy observation that Paley digraphs are weakly distance-regular.

\section{The proof of Theorem~\ref{main1}}

We note that the second statement of Theorem~\ref{main1} follows immediately from Lemma~\ref{skew}. To prove the first statement of Theorem~\ref{main1}, we write $k_{1}=k_{2}=k$. We set $I:=\{i\mid R_{i}\in R_{1}^{2}\}$ and $J:=\{i\mid R_{i}\in R_{1}R_{1^{*}}\}$. Then $I\subseteq\{1,1^{*},2\}$ by (\ref{1.1a}). Since $R_{1}R_{1^{*}}=R_{1^{*}}R_{1}$ from the commutativity of $\mathfrak{X}$, we get $J=J^{*}$. Since $k>1$, we have $\{0\}\subsetneqq J$. It follows from \eqref{1.1b} that
\begin{align}
J=\{0,1,1^{*}\},~\{0,2,2^{*}\}~\textrm{or}~\{0,1,1^{*},2,2^{*}\}.\label{set J}
\end{align}
%By \cite[Chapter \Rmnum{2}, Theorem 4.3]{EB84}, all non-diagonal relations of $\mathfrak{X}$ have the same valency. Without loss of generality, we may assume that $k_{i}=k$ for $1\leq i\leq d$.

\begin{lemma}\label{jiben}
The following hold:
\begin{itemize}
\item [{\rm(i)}] $\sum_{j=0}^{d}p_{1,j}^{l}=k$.

\item [{\rm(ii)}] $p_{i,j}^{l}=p_{l,j^{*}}^{i}$, $l,i\in\{1,1^{*},2,2^{*}\}$.

\item [{\rm(iii)}] $\sum_{l\in I}p_{1,1}^{l}=k$.

\item [{\rm(iv)}] $2\sum_{l\in J\cap\{1,2\}}p_{1,1^{*}}^{l}=k-1$.

\item [{\rm(v)}] $\sum_{\alpha\in I}(p_{1,1}^{\alpha})^{2}=k+2\sum_{\beta\in J\cap\{1,2\}}(p_{1,1^{*}}^{\beta})^{2}$.
\end{itemize}
\end{lemma}
\begin{proof} (i) follows by setting $i=1$ in Lemma~\ref{jb} (iii). (ii) holds by Lemma~\ref{jb} (ii) since $k_{1}=k_{2}$. (iii) follows by setting $i=j=1$ in Lemma~\ref{jb} (i). (iv) follows by setting $i=1$ and $j=1^{*}$ in Lemma~\ref{jb} (i), and using (ii).  (v) follows by setting $i=j=l=1$ and $f=1^{*}$ in Lemma~\ref{jb} (iv), and using (ii).
\end{proof}

We list some consequences of Lemma~\ref{jiben} (ii) as follows:
\begin{align}
p_{1,1^{*}}^{2}&=p_{1^{*},1}^{2}=p_{1,2}^{1}=p_{2,1}^{1}=p_{2^{*},1}^{1}=p_{1,2^{*}}^{1}, \label{2.1b}\\
p_{1,1^{*}}^{1}&=p_{1,1^{*}}^{1^{*}}=p_{1,1}^{1}, \label{2.1a}\\
p_{1,1}^{2}&=p_{1^{*},2}^{1}. \label{2.1d}
\end{align}

In the proof of Theorem~\ref{main1}, we also need the following auxiliary lemmas.

\begin{lemma}\label{3.2}
The following conditions are equivalent:
\begin{itemize}
\item [{\rm(i)}] $1\in I$.

\item [{\rm(ii)}] $p_{1,1}^{1}\neq0$.

\item [{\rm(iii)}] $\{1,1^{*}\}\subseteq J$.
\end{itemize}
\end{lemma}
\begin{proof}
Immediate from \eqref{2.1a}.
\end{proof}

\begin{lemma}\label{R2}
Let $\{R_{i}\mid i\in K\}$ be a subset of $\{R_{i}\}_{i=0}^{d}$. Fix $(x,z)\in R_{2}$. Then
\begin{itemize}
\item [{\rm(i)}] $R_{2}R_{1}\subseteq\{R_{i}\mid i\in K\}$ if and only if $(x,w)\in \bigcup_{i\in K}R_{i}$ for all $w\in R_{1}(z)$.

\item [{\rm(ii)}] $R_{2^{*}}R_{1}\subseteq\{R_{i}\mid i\in K\}$ if and only if $(z,w)\in \bigcup_{i\in K}R_{i}$ for all $w\in R_{1}(x)$.
\end{itemize}
\end{lemma}
\begin{proof}
Note that $p_{2,1}^{i}\neq0$ is equivalent to $p_{i,1^{*}}^{2}\neq0$ by Lemma~\ref{jb} (ii), which in turn is equivalent to $R_{i}(x)\cap R_{1}(z)\neq\emptyset$. Thus, $R_{2}R_{1}\subseteq\{ R_{i}\mid i\in K\}$ is equivalent to $R_{1}(z)\subseteq \bigcup_{i\in K}R_{i}(x)$. This proves (i). The proof of (ii) is similar, hence omitted.
\end{proof}

\begin{lemma}\label{d=4}
If $\{R_{2},R_{2^{*}}\}R_{1}\subseteq\{R_{1},R_{1^{*}},R_{2},R_{2^{*}}\}$, then $d=4$.
\end{lemma}
\begin{proof}
By \eqref{1.1a} and \eqref{1.1b},
\begin{align}
R_{1}^{3}\subseteq\{R_{1},R_{1^{*}},R_{2}\}R_{1}\subseteq\{R_{0},R_{1},R_{1^{*}},R_{2},R_{2^{*}}\}.\nonumber
\end{align}
Then
\begin{align}
R_{1}^{4}\subseteq\{R_{0},R_{1},R_{1^{*}},R_{2},R_{2^{*}}\}R_{1}\subseteq\{R_{0},R_{1},R_{1^{*}},R_{2},R_{2^{*}}\}.\nonumber
\end{align}
It follows from induction that
\begin{align}
R_{1}^{i}&\subseteq\{R_{0},R_{1},R_{1^{*}},R_{2},R_{2^{*}}\}R_{1}~\textrm{for}~i\geq5.\nonumber
\end{align}
Since $R_{1}$ generates $\mathfrak{X}$, we obtain $d=4$.
\end{proof}

\begin{lemma}\label{2.2}
For $x,z\in X$, we have $P_{1,1}(x,z)\times P_{1,1^{*}}(x,z)\subseteq\bigcup_{j\in I\cap J}R_{j}$.
\end{lemma}
\begin{proof}Pick $y\in P_{1,1}(x,z)$ and $y'\in P_{1,1^{*}}(x,z)$. Since $z\in P_{1,1}(y,y')$ and $x\in P_{1^{*},1}(y,y')$, we have $(y,y')\in R_{j}$ for some $j\in I\cap J$.
\end{proof}

In the following, we divide the proof of the first statement of Theorem~\ref{main1} into three subsections according to separate assumptions based on the cardinality of the set $I$.

\subsection{The case $|I|=1$}

By (\ref{1.1a}), we have $I=\{i\}$ for some $i\in\{1,1^{*},2\}$. In view of Lemma~\ref{jiben} (iii), one gets $p_{1,1}^{i}=k$. It follows from Lemma~\ref{jiben} (i) that $p_{1,1^{*}}^{i}=0$. This implies $i\notin J$ and also $i\neq 1$ by (\ref{2.1a}). Thus $1\notin I$, and hence $\{1,1^{*},i,i^{*}\}\cap J=\emptyset$. In view of (\ref{set J}), one has $i=1^{*}$ and $J=\{0,2,2^{*}\}$. By Lemma~\ref{jiben} (iv), we have $p_{1,1^{*}}^{2}=(k-1)/2$. Lemma~\ref{jiben} (v) implies $(p_{1,1}^{1^{*}})^{2}=k+2(p_{1,1^{*}}^{2})^{2}.$ Substituting $p_{1,1}^{1^{*}}$ and $p_{1,1^{*}}^{2}$ into the above equation, we get $k=1$, a contradiction.

\subsection{The case $|I|=2$}

We divide our proof into two cases according to whether the set $I$ contains $1$.

\textbf{Case 1.} $I=\{1,1^{*}\}$ or $I=\{1,2\}$.

By Lemma~\ref{3.2} and \eqref{set J}, one gets $J=\{0,1,1^{*}\}$ or $J=\{0,1,1^{*},2,2^{*}\}$.
It follows from Lemma~\ref{jiben} (iv) and \eqref{2.1a} that $2p_{1,1}^{1}+2p_{1,1^{*}}^{2}=k-1$. In view of Lemma~\ref{jiben} (iii), we have $p_{1,1}^{1}+p_{1,1}^{i}=k$ for some $i\in\{1^{*},2\}$. By Lemma~\ref{jiben} (v), we obtain
\[(p_{1,1}^{1})^{2}+(p_{1,1}^{i})^{2}=k+2(p_{1,1^{*}}^{1})^{2}+2(p_{1,1^{*}}^{2})^{2}.\]
In view of (\ref{2.1a}),  we get
\[(k-p_{1,1}^{1})^{2}=k+(p_{1,1}^{1})^{2}+(k-1-2p_{1,1}^{1})^{2}/2.\]
Then $p_{1,1}^{1}=(k-1)/2$ and $p_{1,1^{*}}^{2}=(k-1-2p_{1,1}^{1})/2=0$. Thus, $J=\{0,1,1^{*}\}$.

We claim $I=\{1,1^{*}\}$. Since $p_{1,1}^{1}\neq0$, there exist elements $x,y,z$ such that $(x,y),(y,z),(x,z)\in R_{1}$. By $z\in P_{1,1^{*}}(x,y)$ and Lemma~\ref{2.2}, we have $P_{1,1^{*}}(x,z)\subsetneqq \bigcup_{j\in I\cap J}P_{1,j^{*}}(x,y)$, which implies $I\cap J=\{1,1^{*}\}$. Thus, our claim is valid.

Note that $R_{1}^{i}=\{R_{0},R_{1},R_{1^{*}}\}$ for $i\geq3$. Since $R_{1}$ generates $\mathfrak{X}$, we have $2=1^{*}$, a contradiction.

\textbf{Case 2.} $I=\{1^{*},2\}$.

In view of Lemma~\ref{3.2} and \eqref{set J}, one has $J=\{0,2,2^{*}\}$. By Lemma~\ref{jiben} (iv), we obtain $p_{1,1^{*}}^{2}=(k-1)/2$.

We claim $p_{1,1}^{2}=1$. Pick $(x,z)\in R_{2}$. Since $I\cap J=\{2\}$, Lemma~\ref{2.2} implies
\[P_{1,1^{*}}(x,z)\subseteq\bigcap_{y\in P_{1,1}(x,z)}P_{1,2^{*}}(x,y).\]
Since $|P_{1,1^{*}}(x,z)|=|P_{1,2^{*}}(x,y)|$ for all $y\in P_{1,1}(x,z)$ by \eqref{2.1b}, we get $P_{1,1^{*}}(x,z)=P_{1,2^{*}}(x,y)$ for all $y\in P_{1,1}(x,z)$. Now suppose $y,y'\in P_{1,1}(x,z)$ are distinct. Since $z\in P_{1,1^{*}}(y,y')$ and $J=\{0,2,2^{*}\}$, we have $(y,y')\in R_{2}\cup R_{2^{*}}$. Then we may assume without loss of generality $(y,y')\in R_{2^{*}}$. However, this is a contradiction since $y\in P_{1,2^{*}}(x,y')=P_{1,2^{*}}(x,y)$.

In view of Lemma~\ref{jiben} (iii), we get $p_{1,1}^{1^{*}}=k-1$. By Lemma~\ref{jiben} (v), one has
\[(p_{1,1}^{1^{*}})^{2}+(p_{1,1}^{2})^{2}=k+2(p_{1,1^{*}}^{2})^{2}.\]
Substituting $p_{1,1}^{1^{*}},p_{1,1}^{2}$ and $p_{1,1^{*}}^{2}$ into the above equation, one gets $k=3$.

We claim
\begin{align}
R_{2}R_{1}&=\{R_{1},R_{2},R_{2^{*}}\}\label{subsec2-1},\\
R_{2^{*}}R_{1}&=\{R_{1},R_{1^{*}},R_{2^{*}}\}\label{subsec2-2}.
\end{align}

Indeed, fix $(x,z)\in R_{2}$, and pick $y\in P_{1,1}(x,z)$. Let $w,w',w''$ be three elements such that $R_{1}(z)=\{w,w',w''\}$.  Since $p_{1,1}^{1^{*}}=2$, we may assume that $w',w''\in P_{1,1}(z,y)$.  In view of (\ref{2.1b}), we get $p_{2,1}^{1}=p_{2^{*},1}^{1}=1$. Then we may assume $w'\in P_{2,1}(x,y)$ and $w''\in P_{2^{*},1}(x,y)$. By $p_{1,1^{*}}^{2}=1$, one has $w\in P_{1,1^{*}}(x,z)$. It follows from Lemma~\ref{R2} (i) that \eqref{subsec2-1} is valid.

Since $z\in P_{2,1}(x,w')$, we have $p_{2,1}^{2}\neq0$. Then there exists an element $y'\in P_{1,2}(x,z)$. By $k=3$, one gets  $R_{1}(x)=\{y,y',w\}$. In view of Lemma~\ref{R2} (ii), \eqref{subsec2-2} is valid.

By Lemma~\ref{d=4}, we have $d=4$.

\begin{remark}
Since $k=3$ and $d=4$, we have $|X|=13$. It follows from \cite[Result 1]{MH96} that $\mathfrak{X}$ is isomorphic to $\textrm{Cyc}(13,4)$.
\end{remark}

\subsection{The case $|I|=3$}

By Lemma~\ref{3.2}, one has $\{0,1,1^{*}\}\subseteq J$. In view of Lemma~\ref{jiben} (v) and \eqref{2.1a}, we obtain
\begin{align}
(p_{1,1}^{1^{*}})^{2}+(p_{1,1}^{2})^{2}=k+(p_{1,1}^{1})^{2}+2(p_{1,1^{*}}^{2})^{2}.\label{J}
\end{align}

Suppose $J=\{0,1,1^{*}\}$. In view of Lemma~\ref{jiben} (iv) and \eqref{2.1a}, we get $p_{1,1}^{1}=(k-1)/2$. Then Lemma~\ref{jiben} (iii) implies $p_{1,1}^{1^{*}}+p_{1,1}^{2}=(k+1)/2$. Since $p_{1,1^{*}}^{2}=0$, from \eqref{J}, one gets $(p_{1,1}^{1^{*}})^{2}+(p_{1,1}^{2})^{2}=(k+1)^{2}/4$. It follows that $p_{1,1}^{1^{*}}p_{1,1}^{2}=0$, a contradiction. Hence, $J=\{0,1,1^{*},2,2^{*}\}$.

Suppose $p_{1,1}^{2}=p_{1,1^{*}}^{2}$.  In view of Lemma~\ref{jiben} (iv) and (\ref{2.1a}), one gets $p_{1,1}^{1}+p_{1,1}^{2}=(k-1)/2$. By Lemma~\ref{jiben} (iii), we have  $p_{1,1}^{1^{*}}=(k+1)/2$. In view of \eqref{J}, we get $(p_{1,1}^{1})^{2}+(p_{1,1}^{2})^{2}=(k-1)^{2}/4$. It follows that $p_{1,1}^{1}p_{1,1}^{2}=0$, a contradiction. Hence, $p_{1,1}^{2}\neq p_{1,1^{*}}^{2}$.

We claim
\begin{align}
R_{2}R_{1}&\subseteq\{R_{1},R_{1^{*}},R_{2},R_{2^{*}}\},\label{subsec3-1}\\
R_{2^{*}}R_{1}&\subseteq\{R_{1},R_{1^{*}},R_{2},R_{2^{*}}\}.\label{subsec3-2}
\end{align}

Indeed, fix $(x,z)\in R_{2}$. Pick an element $w\in R_{1}(z)$. Then $P_{1,1}(x,z)\subseteq(R_{1}\cup R_{1^{*}}\cup R_{2^{*}})(w)$ by \eqref{1.1a}. Suppose first $P_{1,1}(x,z)\cap(R_{1}\cup R_{1^{*}})(w)\neq\emptyset$. Then $(x,w)\in R_{1}\cup R_{1^{*}}\cup R_{2}\cup R_{2^{*}}$ by \eqref{1.1a} and \eqref{1.1b}. Next suppose $P_{1,1}(x,z)\subseteq R_{2^{*}}(w)$. By (\ref{2.1d}), we have $P_{1,1}(x,z)=P_{1^{*},2}(z,w)$. Since $p_{1^{*},1}^{2}=p_{1,1^{*}}^{2}\neq0$ from (\ref{2.1b}), there exists an element $y_{0}\in P_{1^{*},1}(x,z)$. In view of $y_{0}\notin P_{1^{*},2}(z,w)$, one gets $(y_{0},w)\in R_{1}\cup R_{1^{*}}$, which implies that $(x,w)\in R_{1}\cup R_{1^{*}}\cup R_{2}\cup R_{2^{*}}$. Therefore, \eqref{subsec3-1} follows from Lemma~\ref{R2} (i). Note that \eqref{subsec3-1} implies
\begin{align}
R_{1^{*}}R_{2^{*}}\subseteq\{R_{1},R_{1^{*}},R_{2},R_{2^{*}}\}.\label{subsec3-3}
\end{align}

Next, pick an element $w\in R_{1}(x)$. Then $P_{1,1}(x,z)\subseteq(R_{0}\cup R_{1}\cup R_{1^{*}}\cup R_{2}\cup R_{2^{*}})(w)$ by \eqref{1.1b}. Suppose first $P_{1,1}(x,z)\cap(R_{0}\cup R_{1}\cup R_{1^{*}}\cup R_{2})(w)\neq\emptyset$. By \eqref{1.1a},\eqref{1.1b} and \eqref{subsec3-3}, we have $(z,w)\in R_{1}\cup R_{1^{*}}\cup R_{2}\cup R_{2^{*}}$. Next, suppose $P_{1,1}(x,z)\subseteq P_{1,2}(x,w)$. Since $p_{1,1}^{2}\neq p_{1,1^{*}}^{2}$, from (\ref{2.1b}) and (\ref{2.1d}), one gets $p_{1^{*},2}^{1}<p_{1^{*},1}^{2}$. Pick an element $y\in P_{1,1}(x,z)$. Since $(y,w)\in R_{2}$,  there exists an element $x'\in P_{1^{*},1}(y,w)$ such that $x'\notin P_{1^{*},2}(y,z)$. Then $(x',z)\in R_{1}\cup R_{1^{*}}$  by \eqref{1.1a}. It follows that $(z,w)\in R_{1}\cup R_{1^{*}}\cup R_{2}\cup R_{2^{*}}$ by \eqref{1.1b}. By Lemma~\ref{R2} (ii), \eqref{subsec3-2} is valid.

By Lemma~\ref{d=4}, we have $d=4$.

\begin{remark}
We do not know any example occurring in the case $|I|=3$.  If there is an association scheme under this case, $\mathfrak{X}$ will not be pseudocyclic by Lemma~\ref{skew}. Indeed $\mathfrak{X}={\rm Cyc}(13,4)$ satisfies $|I|=2$.
\end{remark}

\section{The proof of Theorem~\ref{Main2}}

In order to prove the necessity part
of Theorem~\ref{Main2}, we first show that the girth of $\Gamma$ is $3$ under the condition that $\Gamma$ is not isomorphic to one of the digraphs in (i) and (ii). We then verify the hypotheses of Theorem~\ref{main1}. Since there are weakly distance-regular digraphs of girth $3$ and arbitrarily large diameter (see \cite[Theorem 1.1 (vii) and (viii)]{YYF}), it is nontrivial to verify the hypotheses of Theorem~\ref{main1}, especially \eqref{1.1a}.

In the proof of Theorem~\ref{Main2}, we also need the following two auxiliary results.

\begin{lemma}[{\cite[Theorem 2.10.5]{AEB98}}]\label{prime}
Let $(X,\{R_{i}\}_{i=0}^{d})$ be a primitive translation scheme with $d\geq2$. If $X$ has some cyclic Sylow subgroup (and in particular if $X$ is cyclic), then
$|X|$ is prime and $(X,\{R_{i}\}_{i=0}^{d})$ is cyclotomic.
\end{lemma}

\begin{lemma}\label{lem:2}
Let $\Gamma$ be a commutative weakly distance-regular digraph, and $(x_{0},x_{1},\ldots,x_{q-1})$ be a circuit consisting of arcs of type $(1,q-1)$,
where $q\geq3$. Suppose $k_{(1,q-1)}=k_{(2,q-2)}$. Let $Y_{i}=P_{(1,q-1),(1,q-1)}(x_{i-1},x_{i+1})$. Then
\[Y_{i}=P_{(2,q-2),(q-1,1)}(x_{i-2},x_{i-1})=P_{(q-1,1),(2,q-2)}(x_{i+1},x_{i+2}),\]
where the indices are read modulo $q$. Moreover, if $q>3$ and $\Gamma$ is a Cayley digraph over an additive group, then $Y_{i}-x_{i-1}=Y_{i+1}-x_{i}$.
\end{lemma}
\begin{proof}
It suffice to show both statements for $i=1$ only. 

In order to prove the first statement, by reversing the orientation, it suffices to prove $Y_{1}=P_{(q-1,1),(2,q-2)}(x_{2},x_{3})$. If $y\in Y_{1}$, then
$(x_{0},y,x_{2},\ldots,x_{q-1})$ is also a circuit consisting of arcs of type $(1,q-1)$,
so $y\in P_{(q-1,1),(2,q-2)}(x_{2},x_{3})$. Hence, $Y_{1}\subseteq P_{(q-1,1),(2,q-2)}(x_{2},x_{3})$.
Since $k_{(1,q-1)}=k_{(2,q-2)}$, we get $p_{(1,q-1),(1,q-1)}^{(2,q-2)}=p_{(q-1,1),(2,q-2)}^{(1,q-1)}$ from Lemma~\ref{jb} (ii).
Thus, equality is forced, and the first statement is valid.

Now suppose that $q>3$ and $\Gamma$ is a Cayley digraph over an additive group. From the first statement, we have 
\begin{align}\label{sec4-1}
Y_{1}=P_{(2,q-2),(q-1,1)}(x_{q-1},x_{0}).
\end{align}
Let $y\in Y_{2}$. Since $q>3$, $(x_{1},y,x_{3},\ldots,x_{q-1},x_{0})$ is a circuit which contains two distinct vertices $y$ and $x_{q-1}$. From the first statement again, one gets 
\begin{align}\label{sec4-2}
P_{(1,q-1),(1,q-1)}(x_{0},y)=P_{(2,q-2),(q-1,1)}(x_{q-1},x_{0}). 
\end{align}
Since $y-x_{1},x_{1}-x_{0}\in\Gamma_{1,q-1}(0)$, we obtain
\begin{align}
x_{0}+y-x_{1}&\in(x_{0}+\Gamma_{1,q-1}(0))\cap(y-\Gamma_{1,q-1}(0))\nonumber\\
&=P_{(1,q-1),(1,q-1)}(x_{0},y)\nonumber\\
&=Y_{1}\nonumber
\end{align}
by \eqref{sec4-1} and \eqref{sec4-2}. This proves $y-x_{1}\in Y_{1}-x_{0}$. Since $y\in Y_{2}$ was arbitrary, one has $Y_{2}-x_{1}\subseteq Y_{1}-x_{0}$. The fact that $|Y_{1}|=|Y_{2}|$ implies $Y_{1}-x_{0}=Y_{2}-x_{1}$. The second statement is also valid.
\end{proof}

\begin{proof}[Proof of Theorem~\ref{Main2}]
Observe that all the digraphs in Theorem~\ref{Main2} (i)--(iii) are
primitive weakly distance-regular circulant digraphs.

Let $\Gamma=\textrm{Cay}(\mathbb{Z}_{p},S)$ be a primitive weakly distance-regular digraph of girth $g$, where $\mathbb{Z}_{p}=\{0,1,\ldots,p-1\}$ is the cyclic group of order $p$, written additively. It is an easy observation that
$\mathfrak{X}=(V\Gamma,\{\Gamma_{\wz{i}}\}_{\wz{i}\in\wz{\partial}(\Gamma)})$ is a translation scheme. By Lemma~\ref{prime}, $p$ is a prime and $\mathfrak{X}$
is also cyclotomic and hence skew-symmetric.

If $\mathfrak{X}$ is a $2$-class association scheme, then $\Gamma$ is isomorphic to a Paley digraph. Suppose that $\Gamma$ is not
the circuit and $|\wz{\partial}(\Gamma)|>3$. We only need to prove that $\Gamma$ is isomorphic to the digraph in (iii). We prove it step by step.

\begin{step}\label{g-1}
~{\rm $S=\Gamma_{(1,g-1)}(0)$.}\vspace{-0.3cm}
\end{step}

Suppose $(1,q)\in\wz{\partial}(\Gamma)$ and pick $x_1\in\Gamma_{(1,q)}(0)$.
By Lemma~\ref{lem:1}, there exists  $s\in\mathbb{Z}_{p}$ such that $\Gamma_{(1,g-1)}^{(s)}=\Gamma_{(1,q)}$, which implies that
there exists $y_1\in\Gamma_{(1,g-1)}(0)$ such that $sy_1=x_1$. Then there exists a circuit
$(y_1,y_2,\dots,y_g=0)$ of length $g$ consisting of arcs of type $(1,g-1)$. Multiplication
by $s$ gives a path of length $g-1$ from $x_1$ to $0$, which implies $(0,x_1)\in\Gamma_{(1,g-1)}$.
This implies $q=g-1$, and we have proved
$\{(1,q)\mid(1,q)\in\wz{\partial}(\Gamma)\}=\{(1,g-1)\}$.
Our claim then follows.

\begin{step}\label{girth}
~{\rm $g=3$, or equivalently, $(1,2)\in\wz{\partial}(\Gamma)$.}\vspace{-0.3cm}
\end{step}

Suppose, to the contrary that $g>3$. Let $(x_{0}=0,x_{1},\ldots,x_{g-1})$ be a circuit of length $g$, where the subscripts of $x$
are read modulo $g$. Since
$\mathfrak{X}$
is a cyclotomic scheme, we have $k_{(1,g-1)}=k_{(2,g-2)}$. In the notation of Lemma~\ref{lem:2} by setting $q=g$, one gets $P_{(g-1,1),(2,g-2)}(x_{i+2},x_{i+3})=Y_{i+1}$. By Lemma~\ref{lem:2} applied to the circuit $(y_{i},x_{i+1},x_{i+2},x_{i+3},\ldots,x_{i-1})$ for any $y_{i}\in Y_{i}$, one obtains $Y_{i}-x_{i-1}=P_{(1,g-1),(1,g-1)}(y_{i},x_{i+2})-y_{i}=P_{(g-1,1),(2,g-2)}(x_{i+2},x_{i+3})-y_{i}=Y_{i+1}-y_{i}$. 
Since $x_j\in Y_j$ for $1\leq j<i$ and $x_0=0$, it follows by induction that
\begin{equation}\label{step2}
Y_{1}=Y_{i+1}-y_{i}\quad(y_i\in Y_i).
\end{equation}

We prove $Y_{1}^{(i)}=Y_{i}$ for $1\leq i\leq g$ by induction on $i$, where $Y_{1}^{(i)}:=\{ih\mid h\in Y_{1}\}$. The case $i=1$ is trivial. Suppose $i\geq1$. By induction, 
we have $Y_{1}^{(i)}=Y_{i}$. Let $y\in Y_1$. Then $iy\in Y_{i}$, so
$Y_{1}=Y_{i+1}-iy$ by
\eqref{step2}. This implies $(i+1)y\in Y_{i+1}$, and so $Y_{1}^{(i+1)}\subseteq Y_{i+1}$. 
The fact $|Y_{1}|=|Y_{i+1}|$ implies $Y_{1}^{(i+1)}=Y_{i+1}$.
% for all $1<i\leq g$. 
Since $0\in Y_{g}=Y_{1}^{(g)}$, one has $gy=0$ for some $y\in Y_{1}$. The fact that $y\neq 0$ implies that $p=g$. Then $\Gamma$ is a circuit, contrary to the assumption. Thus, $g=3$.

\begin{step}\label{bkb}
~{\rm $\Gamma_{(1,2)}^{2}\subseteq\{\Gamma_{(1,2)},\Gamma_{(2,1)},\Gamma_{(2,3)}\}$.}\vspace{-0.3cm}
\end{step}

Since $\mathfrak{X}$ is skew-symmetric, from Steps~\ref{g-1} and \ref{girth}, one obtains 
\[\Gamma_{(1,2)}^{2}\subseteq\{\Gamma_{(1,2)},\Gamma_{(2,1)},\Gamma_{(2,3)},\Gamma_{(2,4)}\}.\] Suppose $(2,4)\in\wz{\partial}(\Gamma)$. By Lemma~\ref{lem:1},
there exists $s\in\mathbb{Z}_{p}$ such that $\Gamma_{(1,2)}^{(s)}=\Gamma_{(2,4)}$.
Pick a circuit $(x_{0},x_{1},x_{2})$ of length $3$.
Then $(x_{0},x_{1}),(x_{1},x_{2}),(x_{2},x_{0})\in\Gamma_{(1,2)}$.
Thus $(sx_{0},sx_{1}),(sx_{1},sx_{2}),(sx_{2},sx_{0})\in\Gamma_{(2,4)}$.
Note that there exist vertices $y_{0},y_{1},y_{2}$ such that
$(sx_{0},y_{0},sx_{1},y_{1},sx_{2},y_{2})$ is a circuit.
Since $(sx_{2},sx_{0})\in\Gamma_{(2,4)}$,
$(sx_{0},y_{0},sx_{1},y_{1},sx_{2})$ is a shortest path from $sx_{0}$ to $sx_{2}$. Then $(sx_{0},y_{0},sx_{1},y_{1})$ is a shortest path from $sx_{0}$ to $y_1$.
Similarly, $(y_{1},sx_{2},y_{2},sx_{0})$ is a shortest path from $y_1$ to $sx_{0}$.
This implies $\wz{\partial}(sx_{0},y_{1})=(3,3)$, contrary to that fact that $\mathfrak{X}$ is skew-symmetric. Hence, $(2,4)\notin\wz{\partial}(\Gamma)$ and the desired result follows.

\begin{step}\label{final}
~{\rm $\Gamma$ is isomorphic to the digraph in Theorem~\ref{Main2} (iii).}\vspace{-0.3cm}
\end{step}

Setting $R_{1}=\Gamma_{(1,2)}$ and $R_{2}=\Gamma_{(2,3)}$, the conditions \eqref{1.1b} and \eqref{1.1c} of Theorem~\ref{main1} hold since $\mathfrak{X}$ is skew-symmetric, while \eqref{1.1a} is satisfied by Step~\ref{bkb}. Since $\mathfrak{X}$ is cyclotomic, it is pseudocyclic. Thus, the second statement of Theorem~\ref{main1} implies that (iii) holds. \end{proof}

\section*{Acknowledgements}
A.~Munemasa is supported by JSPS Kakenhi (Grant No.~17K05155),
K.~Wang is supported by NSFC (11671043),
Y.~Yang is supported by the Fundamental Research Funds for the Central Universities (Grant No.~2652017141).
A part of this research was done while A.~Munemasa was visiting Beijing Normal University,
and while Y.~Yang was visiting Tohoku University. We acknowledge the support
from RACMaS, Tohoku University.

%\end{CJK*}

\end{document}